\documentclass{article}
\usepackage{amsthm, amssymb,comment}
\usepackage{graphicx}
\usepackage{epsfig}
\usepackage{psfrag}

\newcommand{\E}{\mathbb E}
\newcommand{\R}{\mathbb R}
\newcommand{\PP}{\mathbb P}
\newcommand{\sD}{\mathcal D}
\newcommand{\dd}{\partial}

\newcommand{\sL}{\mathcal L}
\newcommand{\sN}{\mathcal N}

\newtheorem{thm}{Theorem}[section]
\newtheorem{lem}[thm]{Lemma}
\newtheorem{prop}[thm]{Proposition}

\newtheorem{rem}[thm]{Remark}

\newtheorem{eg}[thm]{Example}

\begin{document}
\setcounter{footnote}{1}
\title{Fake exponential Brownian motion}

\date{
      October 4, 2012}

\author{David Hobson \\
University of Warwick, Coventry CV4 7AL, UK.}

\maketitle

\begin{abstract}
We construct a fake exponential Brownian motion, a continuous martingale different from 
classical exponential Brownian motion but with the same marginal distributions, thus 
extending results of Albin and Oleszkiewicz for fake Brownian motions. The ideas extend 
to other diffusions.

\end{abstract}

\section{The problem}
\label{sec-intro}

Brownian motion is a martingale with centred Gaussian marginals, and 
variance $t$. Hamza and Klebaner~\cite{HamzaKlebaner:06} asked, are 
there other real-valued martingale processes with the same univariate 
marginals as Brownian motion? Since Dupire~\cite{Dupire:97} has shown 
that Brownian motion is the unique 
martingale diffusion with $\sN(0,t)$ marginals, finding alternative 
martingales with the same marginals involves relaxing either the 
continuity assumption, or the Markov property.

Hamza and Klebaner~\cite{HamzaKlebaner:06} answered their question in 
the positive, but it turns out that Madan and Yor~\cite{MadanYor:02} had 
already provided a construction based on the Az\'{e}ma-Yor solution of 
the Skorokhod embedding problem. Both the Madan and 
Yor~\cite{MadanYor:02} and Hamza and Klebaner~\cite{HamzaKlebaner:06} 
constructions are discontinuous processes. However, 
Albin~\cite{Albin:07} produced an ingenious solution to the problem with 
continuous paths, based on products of Bessel process. Subsequently, 
Oleszkiewicz~\cite{Oleszkiewicz:08} produced a simpler construction 
which is extremely natural and is based on mixing and the Box-Jenkins 
simulation of a Gaussian variable. Oleszkiewicz also introduced the 
evocative term {\em fake Brownian motion} to describe a (continuous) 
martingale with the same univariate marginal distributions as Brownian 
motion.

Now that there has been serious study of fake Brownian motion it seems 
reasonable to ask if there exist fake versions of other canonical martingales. In a 
related context, Hirsch et 
al~\cite{HirschProfetaRoynetteYor:11} give several general methods of constructing 
processes with given marginals, but none is quite 
suited to the problem we consider here.
In particular, in this article we provide a family of fake exponential 
Brownian motions: we give a martingale stochastic process with 
continuous paths and with lognormal marginals matching those of exponential 
Brownian motion, which is not exponential Brownian motion. The ideas 
apply to other time-homogeneous diffusions.

One reason for interest in fake processes in this sense comes from mathematical finance, 
see Hobson~\cite{Hobson:11}. In finance knowledge of a continuum of option prices of a 
fixed maturity (one for each possible strike) is equivalent to knowledge of the marginal 
distribution of the underlying financial asset under the pricing measure. This result 
dates back to Breeden and Litzenberger~\cite{BreedenLitzenberger:78}. Knowledge of the 
prices of options for the continuum of strikes and maturities is equivalent to knowledge 
of all the marginals of the underlying financial asset, but implies nothing about the 
finite-dimensional distributions. (However, general theory tells us that the asset 
price, suitably discounted, must be a martingale, again under the pricing measure.) In 
finance we often have knowledge of option prices, but we have little reason to believe 
in any particular model. The canonical model in finance is the 
Osborne-Samuelson-Black-Scholes exponential Brownian motion model, under which vanilla 
option prices are given by the Black-Scholes formula. The fact that we can provide 
alternative martingale models with the same marginals, and therefore which support the 
same option prices, shows that there is no unique model which is consistent with those 
prices.
 
Let $\delta_x$ denote the unit mass at $x$. Suppose 
$P=(P_t)_{t \geq 0}$ is exponential Brownian motion, scaled such that 
$P_0=1$, and with drift parameter chosen such that $P$ is a martingale. 
Setting volatility equal to one we may assume that $P$ solves $dP_t/P_t 
= dB_t$. Let $(\nu_t)_{t \geq 0}$ be the family of laws of $(P_{t})_{t 
\geq 0}$ with lognormal densities $(p_t(\cdot))_{t > 0}$ where
\begin{equation}
\label{eqn:ebmdensity} p_t(x) = \frac{1}{\sqrt{2 \pi}}
\frac{e^{-t/8}}{t^{1/2}}
\frac{e^{-(\ln x)^2/2t}}{x^{3/2}} . 
\end{equation}
We write $\nu_t \sim {\mathcal L}(P_t) \sim p_t(\cdot)$. Following Hamza 
and Klebaner we ask {\em does there exist a continuous martingale 
with marginals $(\nu_t)_{t \geq 0}$, which is distinct from 
exponential Brownian motion?} The main goal of this note is to show that 
the answer to this question is positive.

For the general diffusion case, let $X=(X_t)_{t \geq 0}$ be a time-homogeneous 
martingale diffusion so that $X$ solves $dX_t = \sigma(X_t)dB_t$ subject to $X_0 =x_0$. 
Let $I
\subseteq \R$ be the state space of $X$ and let 
$\mu^X_t$, with associated density $f_t(\cdot)$, be the law of $X_t$. 
Assume that the boundary points of $I$ are not 
attainable so that $I = (\ell, r)$ where $-\infty \leq l < x_0 < r \leq \infty$ and that 
$\sigma$ is positive and continuously differentiable on the interior of $I$. Further, we 
suppose 
that 
$f_t(x)$ is continuously differentiable in both $x$ and $t$. This is automatically 
satisfied for Brownian motion and exponential Brownian motion and by results of 
Rogers~\cite{Rogers:85} is also satisfied if 
$\sigma$ is thrice continuously differentiable and $\int_\ell 
\frac{1}{\sigma(x)} dx = \infty = \int^r \frac{1}{\sigma(x)} dx$.

Our goal is to describe a fake version of $X$, where a fake version of 
the martingale $X$ is a continuous martingale $\tilde{X} = 
(\tilde{X}_t)_{t \geq 0}$ such that ${\mathcal L}(\tilde{X}_t) \sim 
\mu^X_t$, which is distinct from $X$.

Using Jensen's inequality and the martingale property of $X$ we can 
deduce that the family $(\mu^X_t)_{t \geq 0}$ must be increasing in 
convex order. Further, the results of Kellerer~\cite{Kellerer:72} imply 
that the fact that a family of measures is increasing in convex order is 
necessary and sufficient for there to be a martingale with those 
marginals. Of course, such a process need not be a diffusion. For a 
time-homogeneous diffusion we have the following result, which follows from Tanaka's 
formula (Revuz and
Yor~\cite[Theorem 6.1.2]{RevuzYor:99}), the fact that $\int_0^t I_{
\{ X_s \geq x \} }dX_s$ is a martingale, and an application of Fubini's
Theorem.

\begin{lem}[Carr and Jarrow~\cite{CarrJarrow:87}, Klebaner~\cite{Klebaner:02}]
\label{lem:mgdiff}
Suppose $\sigma(.)$ is continuous. 
Then
\[ \E[(X_T - x)^+] = (x_0-x)^+ + \frac{\sigma(x)^2}{2} \int_0^T f_t(x) 
dt \]
\end{lem}


\section{Fake solutions based on Skorokhod embeddings}
Given a stochastic process $X$ on a state space $E$, and a probability law $\mu$ on $E$, 
the Skorokhod embedding problem~\cite{Skorokhod:65} is to find a stopping time $\tau$ 
such that $X_\tau \sim \mu$. In the simplest Brownian case, the problem becomes, given 
Brownian motion null at zero and a centred, square-integrable probability measure $\mu$ 
on $\R$, to find a stopping time $\tau$ such that $B_\tau \sim \mu$ and $\E[\tau]$ is 
finite (and then necessarily $\E[\tau]= \int x^2 \mu(dx)$).

There are a multiplicity of solutions of the Skorokhod embedding problem 
for Brownian motion. Many of these solutions are based on considering 
the bivariate process $(B_t,A_t)_{t \geq 0}$ where $A$ is some 
increasing additive functional, null at 0. Then the solution of the 
embedding problem is to take $\tau \equiv \tau^A_\mu = \inf\{t>0 : (B_t, 
A_t) \in \sD^A_\mu \}$ where $\sD_\mu$ is some appropriate domain in $\R 
\times \R_+$. In particular the solutions of \{Root~\cite{Root:69}, 
Rost~\cite{Rost:76}, Az\'{e}ma-Yor~\cite{AzemaYor:79a}, 
Vallois~\cite{Vallois:83, Vallois:92}\} are based on the choices $A_\cdot$ 
equals \{time, time, the maximum process, the local time\} 
respectively.

These constructions are suggestive of a direct construction of a 
discontinuous fake martingale diffusion $\tilde{X}= (\tilde{X}_t)_{t \geq 0}$. 
Fix an additive functional $A$ and let $\sD_t \equiv 
\sD^A_{\mu^X_t}$ denote the domain such that $\tau_t \equiv 
\tau^A_{\mu^X_t}$ is a solution of the Skorokhod embedding problem for 
$\mu^X_t$ in Brownian motion. Then provided the stopping times $\tau_t$ 
are non-decreasing in $t$, or equivalently provided the stopping domains 
$\sD_t$ are decreasing in $t$, the process $\tilde{X}$ given by 
$\tilde{X}_t = B_{\tau_t}$ is a martingale with the required 
distributions. All that remains is to check that the regions $\sD_t$ are 
indeed decreasing in $t$.

In general, the fact that a family of distributions $(\mu^X_t)_{t \geq 0}$ 
is increasing in convex order is not sufficient to guarantee that 
the associated domains $\sD_t$ are decreasing in $t$, and for any 
specific additive functional and associated construction of the solution 
to the Skorokhod embedding problem this needs to be checked.

Madan and Yor~\cite{MadanYor:02} use this approach and the Az\'ema-Yor 
solution to describe a discontinuous fake diffusion. The condition that 
$\sD_t$ is decreasing in $t$ is restated as the fact that the marginal 
distributions are increasing in residual mean life order  
(equivalently, the barycentre functions are increasing in $t$). It can 
be checked that this property holds for the lognormal family of 
distributions so that the Madan and Yor~\cite{MadanYor:02} construction 
gives a fake exponential Brownian motion. Similarly, the Root, R\"{o}st, 
and Vallois constructions all extend from the univariate case for a 
single marginal, to give fake exponential Brownian motions. However, in 
all cases the resulting process is discontinuous.

\section{Continuous Fake exponential Brownian motion} 

Let $X = (X_t)_{t \geq 0}$ be the solution of an It\^{o} stochastic 
differential equation, and let $\mu^X_t \sim \sL(X_t)$. Then 
Gy\"{o}ngy~\cite{Gyongy:86} showed that there is a diffusion process $Y= 
(Y_t)_{t \geq 0}$ such that $\sL(Y_t) \sim \mu^X_t$. In a related 
result, Dupire~\cite{Dupire:97} showed that if the family of 
distributions $(\mu_t)_{t \geq 0}$ is increasing in convex order, and 
if the associated call price functional $C(t,x) = \int_{\R}(y-x) 
\mu_t(dy)$ is sufficiently regular (in particular, $C$ is $C^{1,2}$ and 
has certain limiting properties for large $x$) then there is a 
time-inhomogeneous martingale diffusion $Y$ solving $dY_t = \eta(t,Y_t) dW_t$ such 
that $\sL(Y_t) \sim \mu_t$. The coefficient $\eta$ is given by
\( \eta(t,y)^2 = 2\frac{\partial C(t,y)}{\partial t} / 
\frac{\partial^2 C(t,y)}{\partial^2 y} 
. \)
We call the process $Y$ associated with the family $(\mu_t)_{t \geq 0}$ 
the Dupire diffusion. 


Given a process whose univariate marginals coincide with those of 
exponential Brownian motion, both Gy\"ongy and Dupire show how to 
construct a diffusion process whose marginals are lognormal. However, in 
both cases the resulting process is exponential Brownian motion itself.
Instead, in this note we give a family of continuous martingales 
which are not exponential Brownian motion but which share the 
lognormal univariate marginal distributions of exponential Brownian 
motion. The process we construct is a mixture of two well-chosen 
martingales.

Let $X=(X_t)_{t \geq 0}$ be some time-homogeneous martingale diffusion 
such that $X_0=x_0$ and for $t>0$, $\sL(X_t) \sim f_t(\cdot)$. 
The idea is to try to write $f_t(x) = c g_t(x) + (1-c)h_t(x)$ 
for a constant $c \in (0,1)$ and a pair of families 
$\{(g_t)_{t > 0}, (h_t)_{t > 0}\}$ such that for each $t$, $g_t$ 
and $h_t$ are the densities of random variables with mean $x_0$, and such that 
if $\mu^G_t$ (respectively $\mu^H_t$) is the law of a random variable 
with density $g_t$ (respectively $h_t$), and if $\mu^G_0 = \delta_{x_0} = \mu^H_0$, 
then 
the family $(\mu^G_t)_{t 
\geq 0}$ (respectively $(\mu^H_t)_{t \geq 0}$) is increasing in convex 
order.  Then, if $G=(G_t)_{t \geq 0}$ (respectively $H=(H_t)_{t \geq 0}$) 
is the Dupire diffusion  associated with
the family of laws $(\mu^G_t)_{t \geq 0}$ (respectively 
$(\mu^H_t)_{t \geq 0}$) and if $\tilde{X} = (\tilde{X})_{t \geq 0}$ is defined by
\( \tilde{X}_t = G_t I_{ \{ Z^c = 1 \} } + H_t I_{ \{ Z^c=0 \} } \)
then $\tilde{X}$ is a fake version of $X$. Here $Z^c$ is a Bernoulli 
random variable which is independent of $G$ and $H$, is known at time 0, 
and is such that $\PP(Z^c=1)=c$, where $c \in (0,1)$.

It remains to show that we can construct the families $\{(g_t)_{t > 0},(h_t)_{t 
> 0} \}$ 
and that the resulting process is not $X$ itself. For this we need 
the family $(\mu^G_t)_{t \geq 0}$ to be increasing (in convex order), 
but not too 
quickly as that will mean there is no `room' for the family 
$(\mu^H_t)_{t \geq 0}$ 
to be increasing. 

The final simplifying idea is to suppose that $G$ is a time change of 
$X$, $G_t = X_{a(t)}$. By choosing $a$ to be an increasing, 
differentiable process we automatically get the existence 
of $G$. If further we require that $\dot{a}<1$ we can hope that the 
family $(\mu^H_t)_{t \geq 0}$ is also increasing in convex order.  

\begin{prop}
\label{prop:convexorder}
Suppose $a(t)$ is a strictly increasing, twice continuously differentiable 
function, null at 
0, with
$\dot{a}(t) < 1$ for $t>0$, and define 
\[ K = \inf_{t > 0} \inf_{y \in I} 
\frac{f_t(y)}{f_{a(t)}(y)}. \]
Suppose $K>0$. Fix $c \in (0,K)$ 
and for $t > 0$ set 
\[ h_t(y)= \frac{1}{1-c} \left\{ f_t(y) - cf_{a(t)}(y) \right\}. 
\]
Then $(h_t)_{t > 0}$ is a family of densities which is 
increasing in convex order.

Moreover, there exists a martingale diffusion $H$ such that $\sL(H_t) \sim h_t(\cdot)$.
\end{prop}

\begin{proof}
From the definition of $c$ we have that $h_t$ is non-negative and it is
immediate that $h_t(y)$ integrates to one so that $h_t$ is the density
of a continuous random variable. Further, since $f_t$ corresponds to a
mean $x_0$ random variable, so does $h_t$.

Let $(\tilde{H}_t)_{t > 0}$ be a family of random variables such that $\sL(\tilde{H}_t) 
\sim h_t( \cdot)$. 
By Lemma~\ref{lem:mgdiff} and the hypotheses of the proposition,
\begin{eqnarray}
\frac{\dd}{\dd t} \E[(\tilde{H}_t - x)^+] & = & 
\frac{1}{1-c} \frac{\dd}{\dd t} \E \left[(X_t - x)^+ - c(X_{a(t)} - x)^+ \right]
\label{eqn:Cdot}
\\
& = & \frac{1}{1-c} \sigma(x)^2 \left( f_t(x) - c \dot{a}(t) f_{a(t)}(x) \right)
> 0 . \nonumber 
\end{eqnarray}  
Hence the family 
$(h_t)_{t > 0}$ is a family of densities which is
increasing in convex order.

Define
\begin{equation}   
\label{eqn:etadef}
\eta(t,y)^2 = \sigma(y)^2 \frac{ f_t(y) - c \dot{a}(t)
f_{a(t)}(y)}{f_t(y) - c
f_{a(t)}(y) } .
\end{equation}
and let $H$ be a weak solution of $dH_t = \eta(t,H_t) dB_t$ subject to $H_0=x_0$. By our 
assumptions the 
solution to this SDE is unique in law. Moreover we have
\[ 1 \leq \frac{\eta(t,y)^2}{\sigma(y)^2}
< \frac{f_t(y)}{f_t(y) - cf_{a(t)}(y) } \leq L^2 \]
where $L^2 = K/(K-c)$.
Further, $X^L$ given by $X^L_t = X_{L^2 t}$ is a
martingale and
solves $dX^L_{t} =
\sigma_L(X^L_{t}) dB_t$ where $\sigma_L(x) = L \sigma(x)$.
Then, $\eta(t,y) \leq \sigma_L(y)$ and by Theorem~3 of Hajek~\cite{Hajek:85}, we can
find an increasing time-change $\Gamma$ with $\Gamma_t \leq t$ such that
$\hat{H}_t = X^L_{\Gamma_t}$ solves $d\hat{H}_t = \eta(t,\hat{H}_t) dB_t$. 
Hence, by the uniqueness in law of solutions to this equation, $H$ is a martingale.

Suppose that $X$ is non-negative (the case where $X$ is bounded above or below reduces 
to this case after a reflection and/or a shift.)
Define $C_h:[0,\infty) \times [0,\infty) \mapsto \R$ via $C_h(0,y)=(x_0-y)^+$ and for 
$t > 0$, 
$C_h(t,y) = \int (z-y)^+ h_t(z) dz$. 
Then, using (\ref{eqn:Cdot}) we deduce that $C_h$ solves Dupire's equation
\begin{equation}
\label{eqn:dupire}
 \frac{1}{2} \eta(t,y)^2  C''(t,y) = \dot{C}(t,y) 
\end{equation}
subject to
\begin{equation}  
\label{eqn:bcR+}
 \mbox{$C(0,y) = (x_0 - y)^+$, $C(0,t) = x_0$, $0 \leq C \leq x_0$.}
\end{equation} 
Indeed, by an argument similar to that in 
Ekstr\"{o}m and
Tysk~\cite[Step 5 of the proof of Theorem 2.2]{EkstromTysk:12} $C_h$ is the unique 
solution of (\ref{eqn:dupire}) satisfying (\ref{eqn:bcR+}). 

However, by the results of Dupire~\cite{Dupire:97} and 
Klebaner~\cite{Klebaner:02} $C_H(t,y) = 
\E[(H_t-y)^+]$ also solves (\ref{eqn:dupire}) and (\ref{eqn:bcR+}) and 
hence, $\sL(H_t) \sim h_t(\cdot)$.

Now consider the case where $X$ the range of $X$ is the whole real line. Let $\eta$ and 
$H$ be as before. Fix $\overline{T} \in (0,\infty)$.
Let $C_h: [0,\overline{T}] \times \R \mapsto \R$ be the 
solution of (\ref{eqn:dupire}) subject to
\begin{equation}
\label{eqn:bcR}  
 \mbox{$C(0,y) = (x_0 - y)^+$, $\sup_{t \leq \overline{T}} \lim_{y \uparrow \pm \infty}  
C_t(t,y) - (x_0-y)^+ = 0$, $(x_0-y)^+ \leq C_h(t,y) \leq (x_0 - y)^+ + J$, }
\end{equation}
where $J = J(\overline{T}) = \frac{\E[(X_{L^2 \overline{T}} - x_0)^+]}{1-c}$.
By a small
modification of the arguments of Ekstr\"{o}m and
Tysk~\cite{EkstromTysk:12} to allow for the fact that we are working with real-valued 
processes, it follows that the solution to this equation is unique.
But $C_H(t,y) =
\E[(H_t-y)^+]$ also solves (\ref{eqn:dupire}) and $(x_0 - y)^+ \leq C_H(t,y) \leq 
\E[(X_{L^2 t} - y)^+]$ and hence $\sL(H_t) \sim h_t$ for $t \leq \overline{T}$. Since 
$\overline{T}$ is arbitrary, the result follows for all $t$.
\end{proof}

\begin{thm}
Let $G$ be given by $G_t = X_{a(t)}$ 
and let $H$ 
be the Dupire diffusion associated with the 
family of densities $(h_t)_{t > 0}$. Then
\[ \tilde{X}_t = G_t I_{ \{ Z^c = 1 \} } + H_t I_{ \{ Z^c=0 \} } \]
is a fake version of $X$.
\end{thm}

\begin{proof}
If $\tilde{f}_t(x)$ denotes the time-$t$ density of $\tilde{X}$ we have that
\[ \tilde{f}_t(x) = c g_{t}(x) + (1-c) h_t(x) =c f_{a(t)}(x) + (1-c) h_t(x) = f_t(x). 
\]
and $X_t$ has the same distribution as $X$. 

Since $G$ and $H$ are martingales, so is $\tilde{X}$.
But the quadratic variations of $X$ and $\tilde{X}$ are different and hence
$\tilde{X}$ is a fake version of $X$.
\end{proof}

\begin{eg}[Fake Brownian motion]
In this case 
\[ \frac{f_t(y)}{f_{a(t)}(y)} = \sqrt{ \frac{a(t)}{t} } e^{y^2/2 [ 
1/a(t)-1/t]} \geq \sqrt{ \frac{a(t)}{t} }. \]
Fix $c \in (0,1)$ and choose $a(t)/t > c^2$ with $\dot{a}(t) \leq 1$, 
for 
example ${a}(t)=K^2 t$ for some $K \in (c,1)$.
\end{eg}

\begin{eg}[Fake exponential Brownian motion]
Recall the density given in (\ref{eqn:ebmdensity}). Then
\[ \min_{x>0} \frac{f_t(x)}{f_{a(t)}(x)} = 
\min_{x>0} \sqrt{ \frac{a(t)}{t} } e^{[a(t)-t]/8} e^{(\ln x)^2[ 
1/a(t)-1/t]/2} = \sqrt{ \frac{a(t)}{t} } e^{[a(t)-t]/8} = 
\frac{\psi(a(t))}{\psi(t)} \]
where $\psi$ is the increasing function $\psi(t) = \sqrt{t}e^{t/8}$. 
Fix $K \in (0,1)$ 
and set $a(t)= \psi^{-1}(K \psi(t))$. Then $a(t)<t$ and
\( \dot{a}(t) = \phi(t)/\phi(a(t)) \) 
where $\phi$ is the decreasing function $\phi(t) = \frac{t+4}{8t}$,
so that $\dot{a}<1$ for $t>0$.
\end{eg}

\begin{rem}
The argument can be extended to cover the case where the original 
process $X$ is a time-inhomogeneous martingale diffusion $dX_t 
= \sigma(t,X_t) dB_t$ provided $k>0$ where
\[ k := \inf_{t > 0} \inf_{y \in I} \left\{ \frac{f_t(y)}{f_{a(t)}(y)} \wedge 
\frac{f_t(y) \sigma(t,y)^2}{f_{a(t)}(y) \sigma(a(t),y)^2} \right\} .
\]
\end{rem}



\end{document}